\documentclass[letterpaper, 10 pt, conference]{ieeeconf}  

\usepackage{booktabs}
\usepackage{lineno}
\usepackage[cmex10]{amsmath}
\usepackage{amscd,amsfonts,amsmath,amssymb,mathrsfs,pifont,stmaryrd,tipa}
\usepackage{array,cases,dsfont,graphicx,texdraw}
\usepackage{wrapfig}
\usepackage{graphics} 
\usepackage{epsfig} 
\usepackage{stfloats}
\usepackage{subfig}
\usepackage{hyperref}
\usepackage{color}

\newtheorem{Remark}{Remark}

\newtheorem{Definition}{Definition}

\newenvironment{Proof}{\noindent{\em Proof:\/}}{\hfill $\Box$\par}
\newtheorem{Theorem}{Theorem}

\newtheorem{Assumption}{Assumption}

\title{\LARGE \bf
Adaptive Approaches for Fully Distributed Nash Equilibrium Seeking in Networked Games}

\author{Maojiao Ye and Guoqiang Hu
\thanks{M. Ye is with the School of Automation, Nanjing University of Science and Technology, Nanjing 210094, P.R. China (Email: mjye@njust.edu.cn).
G. Hu is with the School of Electrical and Electronic Engineering, Nanyang Technological University, 639798, Singapore (Email: gqhu@ntu.edu.sg).
}
\thanks{This work is supported by the Natural Science Foundation of China (NSFC) under Grant 61803202, the Natural Science Foundation of Jiangsu Province, No. BK20180455 and the Fundamental Research Funds for the Central Universities, No. 30918011332.}
}
\begin{document}

\maketitle
\thispagestyle{empty}
\pagestyle{empty}

\begin{abstract}
This paper considers the design of fully distributed Nash equilibrium seeking strategies for multi-agent games. To develop fully distributed seeking strategies, two adaptive control laws, including a node-based control law and an edge-based control law, are proposed. In the node-based adaptive strategy, each player adjusts their own weight on their procurable  consensus error dynamically. Moreover, in the edge-based algorithm, the fully distributed strategy is designed by adjusting the weights on the edges of the communication graph adaptively. By utilizing LaSalle's invariance principle, it is shown that the Nash equilibrium is globally asymptotically stable by both strategies given that the players' objective functions are twice-continuously differentiable, the partial derivatives of the players' objective functions with respect to their own actions are globally Lipschitz, the pseudo-gradient vector of the game is strongly monotone and the communication network is undirected and connected. In addition, we further show that the edge-based method can be easily adapted to accommodate time-varying communication conditions, in which the communication network is switching among a set of undirected and connected graphs. In the last, a numerical example is given to illustrate the effectiveness of the proposed methods.
\end{abstract}


\section{INTRODUCTION}
Distributed Nash equilibrium seeking for multi-agent games becomes a thriving research topic in recent years (see e.g., \cite{YECYBER2017}-\cite{YECDC19} for more comprehensive literature review). In particular, continuous-time consensus-based algorithms are shown to be very powerful and effective methods that can solve networked games with incomplete information. For example, networked aggregative games can be addressed by combing dynamic average consensus protocols with the gradient methods \cite{YECYBER2017}\cite{LiangAuto2017}\cite{Dentnnls}. Adaptations to general multi-agent games can be achieved by utilizing gradient search algorithms and leader-following consensus protocols \cite{YETAC2017}\cite{YECYBER2018}\cite{LuTcyber}\cite{YECDC19}. Likewise, $N$-cluster games were solved in \cite{YEAUTO2018}-\cite{YETAC2019} by taking the advantages of average consensus and the gradient search. Roughly speaking, the main idea behind the design of the consensus-based Nash equilibrium seeking algorithms in many existing works can be summarized as follows.
\begin{enumerate}
  \item Employ consensus protocols, either leader-following consensus protocols or dynamic average consensus protocols, to estimate information required for the optimization of the players' objective functions.
  \item Based on the estimated information, each player updates its own action by utilizing the gradient search based optimization techniques.
\end{enumerate}

Most importantly, to ensure the convergence of the closed-loop system, one should include a control gain (singular perturbation parameter) to ensure that the consensus part would be faster than the optimization part. In this way, the consensus-based algorithms are actually in two-time scales and they have advantages due to their simplicity and effectiveness. However, one drawback for this class of methods is that the explicit quantification of the control gain (singular perturbation parameter) depends on the Lipschitz constants of the players' objective functions, the topologies of the communication graphs and the number of players in the game. It should be noted that in many practical situations, these centralized information can hardly be obtained by all the players for games under distributed communication networks. Moreover, as the choices of the control gain (singular perturbation parameter) depend on the Lipschitz constants as well as the network topology, these methods loss plug-in plug-out property. More specifically, if there exist players that join the game or leave the game, the singular perturbation parameter should be re-tuned to ensure the stability of the new Nash equilibrium.

Actually, in \cite{LuTcyber}, the authors tried to adopt a decaying control gain which is not integrable but square integrable to remove the centralized control gain. Unfortunately, the decaying control gain is shared among all the players' control strategies, indicating that the players need to coordinate on it.
In fact, the authors in \cite{KoshalOR16}
showed that in the discrete-time scenario, exact convergence to the Nash equilibrium can be obtained by utilizing the method proposed therein if coordinated stepsize is adopted while only convergence to a neighborhood of the Nash equilibrium can be concluded if the stepsize is uncoordinated. Similarly, the method in \cite{Dentnnls} not only requires the determination of the singular perturbation parameter but also some other shared control gains. In
\cite{LiangAuto2017}, the authors adopted several coordinated control gains  among all the players' seeking strategies, whose quantifications also require centralized information such as the number of players and the maximum value of the players' objective functions. \textbf{In short, most of the existing literature on continuous-time consensus-based Nash equilibrium seeking strategies require the determinations of some coordinated control gains, which possibly depend on centralized information, to ensure the convergence of the closed-loop system.} In this regard, these methods are not fully distributed.  Motivated by the above observations, this paper intends to design fully distributed Nash equilibrium seeking strategies in which the control gains are distributed and are independent of  any centralized information.

To establish fully distributed Nash equilibrium seeking strategies, the main idea of this paper is to adjust the control gain (singular perturbation parameter) adaptively. In comparison with the existing works, the main contributions of the paper are summarized as follows.
\begin{enumerate}
  \item A node-based adaptive strategy and an edge-based adaptive strategy are proposed to achieve the fully distributed seeking of Nash equilibrium in multi-agent games. The node-based adaptive strategy achieves the goal by adaptively updating each player's weight on its overall procurable consensus error. Moreover, the edge-based adaptive strategy achieves fully distributed Nash equilibrium seeking through adaptively adjusting the weights on the edges of the communication graphs.
  \item Based on the LaSalle's invariance principle, it is theoretically shown that the Nash equilibrium is globally asymptotically stable under both adaptive seeking strategies given that the players' objective functions are twice-continuously differentiable, the partial derivatives of the players' objective functions with respect to their own actions are globally Lipschitz, the pseudo-gradient vector of the game is strongly monotone and the communication network is undirected and connected.
  \item Extensions to time-varying communication topologies where the communication network switches among a set of undirected and connected graphs is discussed for the edge-based Nash equilibrium seeking algorithm. It is shown that the edge-based algorithm can be easily adapted to accommodate time-varying communication graphs.
\end{enumerate}

We proceed the remainder of the paper as follows.  In Section \ref{NP}, the motivation of the paper is explained and the problem considered in this paper is given. Section \ref{main} presents the main results of the paper in which a node-based seeking strategy and an edge based seeking strategy are given. Extensions to time-varying communication networks are discussed in Section \ref{time}. Numerical simulations are provided in Section  \ref{num_ex} and conclusions are drawn in Section \ref{conc}.

\section{Motivation and Problem Statement}\label{NP}
Consider a game with $N$ players who are equipped with a communication graph. Denote the player set as $\mathcal{V}=\{1,2,\cdots,N\}.$ Correspondingly, the players are labeled from $1$ to $N$, respectively. In the considered game, each player tries to
\begin{equation}
\text{min}_{x_i}\ \ f_i(\mathbf{x}),
\end{equation}
where $f_i(\mathbf{x})$ and $x_i\in\mathbb{R}$ are the objective function and action of player $i,$ respectively and $\mathbf{x}=[x_1,x_2,\cdots,x_N]^T$. Suppose that each player's own action and objective function are available for itself but not for the other players' actions. Then, the players can seek the Nash equilibrium according to \cite{YETAC2017}
\begin{equation}\label{eq_ref}
\begin{aligned}
\dot{x}_{i}&=-\nabla_if_i(\mathbf{y}_i),\\
\dot{y}_{ij}&=-\theta_{ij}\left(\sum_{k=1}^Na_{ik}(y_{ij}-y_{kj})+a_{ij}(y_{ij}-x_j)\right),
\end{aligned}
\end{equation}
where $\theta_{ij}=\theta \bar{\theta}_{ij}$, $\theta$ is the singular perturbation parameter that would affect the convergence of the method in \eqref{eq_ref}, $\bar{\theta}_{ij}$ is a fixed positive constant, $\nabla_if_i(\mathbf{y}_i)=\frac{\partial f_i(\mathbf{x})}{\partial x_i}|_{\mathbf{x}=\mathbf{y}_i}$ and $\mathbf{y}_i=[y_{i1},y_{i2},\cdots,y_{iN}]^T$. In addition, $a_{ij}$ is $(i,j)$th entry of the adjacency matrix of communication graph, which will be defined later.

In \cite{YETAC2017}, we have shown that there exists a positive constant $\theta^*$ such that for each $\theta\in(\theta^*,\infty),$ the Nash equilibrium is globally exponentially stable by utilizing \eqref{eq_ref} under certain conditions. However, the quantification of $\theta^*$ depends on the Lipschitz constants of the players' objective functions, the topology of the communication network and the number of players in the game. Hence, it is challenging for each player to distributively calculate $\theta^*$. Moreover, if there exist players that join or leave the game, $\theta^*$ might need to be re-tuned to ensure the stability of the resulting new Nash equilibrium.  Motivated by the incentive to relax these restrictions, this paper intends to design adaptive control laws to achieve \textbf{fully distributed} seeking of the Nash equilibrium, whose definition is given as follows.
\begin{Definition}
An action profile $\mathbf{x}^*=(x_i^*,\mathbf{x}_{-i}^*)$ is a Nash equilibrium if
for all $x_i\in\mathbb{R},i\in\mathcal{V},$
\begin{equation}
f_i(x_i^*,\mathbf{x}_{-i}^*)\leq f_i(x_i,\mathbf{x}_{-i}^*),
\end{equation}
where $\mathbf{x}_{-i}=[x_1,x_2,\cdots,x_{i-1},x_{i+1},\cdots,x_N]^T.$
\end{Definition}

The main results of the paper are established based on the following conditions.
\begin{Assumption}\label{ass1}
The players' objective functions $f_i(\mathbf{x})$ for $i\in\mathcal{V}$ are twice-continuously differentiable.
\end{Assumption}

\begin{Assumption}\label{ass2}
For each $i\in\mathcal{V},$ $\nabla_i f_i(\mathbf{x})$ is globally Lipschitz, i.e., there exists a positive constant $l_i$ such that
\begin{equation}
\left|\left|\nabla_i f_i(\mathbf{x})-\nabla_i f_i(\mathbf{z})\right|\right|\leq l_i||\mathbf{x}-\mathbf{z}||,
\end{equation}
for all $\mathbf{x},\mathbf{z}\in\mathbb{R}^N.$
\end{Assumption}
\begin{Assumption}\label{ass3}
There exists a positive constant $m$ such that
\begin{equation}
(\mathbf{x}-\mathbf{z})^T(\mathcal{P}(\mathbf{x})-\mathcal{P}(\mathbf{z}))\geq m||\mathbf{x}-\mathbf{z}||^2,
\end{equation}
for all $\mathbf{x},\mathbf{z}\in\mathbb{R}^N,$ where $\mathcal{P}(\mathbf{x})=[\nabla_1 f_1(\mathbf{x}),\nabla_2 f_2(\mathbf{x}),\cdots,\nabla_N f_N(\mathbf{x})]^T.$
\end{Assumption}
\begin{Remark}
Note that Assumptions \ref{ass1}-\ref{ass3} are standard assumptions in the existing literature (see, e.g., \cite{YETAC2017}\cite{YECYBER2018}). Assumption \ref{ass1} ensures the existence of continuously differentiable gradient vectors, which further indicates the existence of solutions under the proposed methods. Assumption \ref{ass2} can be removed with the presented results sacrificed to be semi-global versions. In addition, Assumption \ref{ass3} ensures the existence of the unique Nash equilibrium on which $\mathcal{P}(\mathbf{x}^*)=\mathbf{0}_N,$ where $\mathbf{0}_N$ denotes an $N$-dimensional zero column vector. Note that $\mathcal{P}(\mathbf{x})=\mathbf{0}_N$ if and only if $\mathbf{x}=\mathbf{x}^*$ by Assumption \ref{ass3} \cite{YECDC19}.
\end{Remark}
\section{Main Results}\label{main}
To achieve fully distributed Nash equilibrium seeking, we suppose that the players are equipped with a communication graph represented by $\mathcal{G}=\{\mathcal{V},\mathcal{E}\},$ where $\mathcal{V}$ and $\mathcal{E}\subset \mathcal{V}\times\mathcal{V}$ are the vertex set and  edge set, respectively. The adjacency matrix associated with $\mathcal{G}$ is defined as $\mathcal{A}=[a_{ij}]$ where $a_{ij}=1$ if $(j,i)\in\mathcal{E}$ and $a_{ij}=0$ if $(j,i)\notin\mathcal{E}.$ The network is said to be undirected if $a_{ij}=a_{ji}$ for all $i,j\in\mathcal{V}$. Correspondingly, the Laplacian matrix of $\mathcal{G}$ is $\mathcal{L}=\mathcal{D}-\mathcal{A}$, where $\mathcal{D}$ is a diagonal matrix whose elements are $\sum_{j=1}^Na_{ij}$ for $i\in\mathcal{V},$ successively. Moreover, the network is connected if there is a path between any two vertices. In addition, we consider that the communication network does not contain any self-loop, i.e., $a_{ii}=0,$ for all $i\in\mathcal{V}.$ With the defined communication graph, the following assumption will be utilized to establish the results in this section.

\begin{Assumption}\label{ass4}
The undirected communication graph $\mathcal{G}$ is connected.
\end{Assumption}

Based on the connectivity of the communication network, two fully distributed Nash equilibrium seeking algorithms will be proposed successively in this section. More specifically, a node-based adaptive design and an edge-based adaptive design will be presented, successively.

\subsection{Fully distributed strategy through node-based design}

In this section, we propose a node-based adaptive strategy to achieve the fully distributed Nash equilibrium seeking.

Motivated by  \cite{YETAC2017} and \cite{LiTAC15}, each player $i$, where $i\in\mathcal{V}$, can update its action according to
\begin{equation}\label{met}
\begin{aligned}
\dot{x}_{i}=&-\nabla_i f_i(\mathbf{y}_i),\\
\dot{y}_{ij}=&-\theta_{ij}\left(\sum_{k=1}^Na_{ik}(y_{ij}-y_{kj})+a_{ij}(y_{ij}-x_j)\right),\\
\dot{\theta}_{ij}=&\gamma_{ij} \left(\sum_{k=1}^Na_{ik}(y_{ij}-y_{kj})+a_{ij}(y_{ij}-x_j)\right)^T\\
&\times \left(\sum_{k=1}^Na_{ik}(y_{ij}-y_{kj})+a_{ij}(y_{ij}-x_j)\right),
\end{aligned}
\end{equation}
where $\gamma_{ij}$ is a positive constant.
\begin{Remark}
Similar to the method in \cite{YETAC2017}, $y_{ij}$ in \eqref{met} stands for player $i$'s local estimation on $x_j.$ Hence, $\mathbf{y}_i$ is generated by player $i$ to estimate $\mathbf{x}$, which is needed in the optimization of the players' objective functions. However, different from \cite{YETAC2017}, the control gain $\theta_{ij}$ is designed to be adaptive. As the time derivative of $\theta_{ij}$ is non-negative, it is clear that $\theta_{ij}$ would be non-decreasing.
\end{Remark}

To establish the convergence results, define the consensus error as $e_{ij}=y_{ij}-x_j,$ then,
\begin{equation}\label{method}
\begin{aligned}
\dot{x}_{i}=&-\nabla_i f_i(\mathbf{e}_{i}+\mathbf{x}),\\
\dot{e}_{ij}=&-\theta_{ij}\left(\sum_{k=1}^Na_{ik}(e_{ij}-e_{kj})+a_{ij}e_{ij}\right)-\dot{x}_j,\\
\dot{\theta}_{ij}=&\gamma_{ij} \left(\sum_{k=1}^Na_{ik}(e_{ij}-e_{kj})+a_{ij}e_{ij}\right)^T\\
&\times \left(\sum_{k=1}^Na_{ik}(e_{ij}-e_{kj})+a_{ij}e_{ij}\right),
\end{aligned}
\end{equation}
where $\mathbf{e}_i=[e_{i1},e_{i2},\cdots,e_{iN}]^T.$

Let
\begin{equation}
\begin{aligned}
\mathcal{A}_0=&\text{diag}\{a_{ij}\},\\
\mathbf{\theta}=&\text{diag}\{\theta_{ij}\},\\
M=&\mathcal{L}\otimes I_{N\times N}+\mathcal{A}_0,\\
\mathbf{e}=&[\mathbf{e}_1^T,\mathbf{e}_2^T,\cdots,\mathbf{e}_N^T]^T,\\
\mathbf{y}=&[\mathbf{y}_1^T,\mathbf{y}_2^T,\cdots,\mathbf{y}_N^T]^T,\\
\end{aligned}
\end{equation}
where $\text{diag}\{p_{ij}\}$ for $i,j\in\mathcal{V}$ defines a diagonal matrix whose elements are $p_{11},p_{12},\cdots,p_{1N},p_{2N},\cdots,p_{NN},$ successively. Moreover, $\otimes$ is the Kronecker product and $I_{N\times N}$ is an $N\times N$-dimensional identity matrix.
Then, the following theorem can be obtained.

\begin{Theorem}\label{the1}
Suppose that Assumptions \ref{ass1}-\ref{ass4} hold. Then,
\begin{enumerate}
  \item The Nash equilibrium $\mathbf{x}^*$ is globally asymptotically stable;
  \item The estimation error $||\mathbf{y}-\mathbf{1}_N\otimes \mathbf{x}||\rightarrow 0,$ as $t\rightarrow\infty$;
  \item For each $i,j\in\mathcal{V}$, $\theta_{ij}$ converges to some finite value.
\end{enumerate}
\end{Theorem}
\begin{proof}
Define the Lyapunov candidate function as
\begin{equation}
V=V_1+V_2+V_3,
\end{equation}
where
\begin{equation}
V_1=\mathbf{e}^TM\mathbf{e},
\end{equation}
\begin{equation}
V_2=\frac{1}{2}(\mathbf{x}-\mathbf{x}^*)^T(\mathbf{x}-\mathbf{x}^*),
\end{equation}
and
\begin{equation}
V_3=\sum_{i=1}^N\sum_{j=1}^N \frac{(\theta_{ij}-\theta_{ij}^*)^2}{\gamma_{ij}}.
\end{equation}

Then,
\begin{equation}
\begin{aligned}
\dot{V}_1=&\dot{\mathbf{e}}^TM\mathbf{e}+\mathbf{e}^TM\dot{\mathbf{e}}\\
=&-2\mathbf{e}^TM \mathbf{\theta}M \mathbf{e}-2\mathbf{e}^TM(\mathbf{1}_N\otimes \dot{\mathbf{x}}),
\end{aligned}
\end{equation}
and
\begin{equation}
\begin{aligned}
\dot{V}_2=&-(\mathbf{x}-\mathbf{x}^*)^T\left[\nabla_i f_i(\mathbf{y}_i)\right]_{vec}\\
=&-(\mathbf{x}-\mathbf{x}^*)^T\left[\nabla_i f_i(\mathbf{x})\right]_{vec}\\
&-(\mathbf{x}-\mathbf{x}^*)^T\left(-\left[\nabla_i f_i(\mathbf{x})\right]_{vec}+\left[\nabla_i f_i(\mathbf{y}_i)\right]_{vec}\right)\\
\leq &-m||\mathbf{x}-\mathbf{x}^*||^2+\max_{i\in \mathcal{V}}\{l_i\}||\mathbf{x}-\mathbf{x}^*||||\mathbf{e}||,
\end{aligned}
\end{equation}
where the notation $[p_i]_{vec}$ for $i\in\mathcal{V}$ defines a column vector with its $i$th entry being $p_i$ and $\max_{i\in \mathcal{V}}\{l_i\}$ is the largest value of $l_i$ for $i\in\mathcal{V}.$ In addition, the last inequality is obtained based on Assumptions \ref{ass2}-\ref{ass3}.

Moreover,
\begin{equation}
\begin{aligned}
\dot{V}_3=&2\sum_{i=1}^N\sum_{j=1}^N(\theta_{ij}-\theta_{ij}^*) (\sum_{k=1}^Na_{ik}(e_{ij}-e_{kj})+a_{ij}e_{ij})^T\\
 &\times (\sum_{k=1}^Na_{ik}(e_{ij}-e_{kj})+a_{ij}e_{ij})\\
= & 2\mathbf{e}^TM(\mathbf{\theta}-\mathbf{\theta}^*)M\mathbf{e},
\end{aligned}
\end{equation}
where $\mathbf{\theta}^*=\text{diag}\{\theta_{ij}^*\}.$

Therefore,
\begin{equation}
\begin{aligned}
\dot{V}=&\dot{V}_1+\dot{V}_2+\dot{V}_3\\
\leq &-2\mathbf{e}^TM\mathbf{\theta}^*M\mathbf{e}-m||\mathbf{x}-\mathbf{x}^*||^2\\
&+\max_{i\in\mathcal{V}}\{l_i\}||\mathbf{x}-\mathbf{x}||||\mathbf{e}||-2\mathbf{e}^TM \mathbf{1}_N\otimes \dot{\mathbf{x}}\\
\leq &-2\mathbf{e}^TM\mathbf{\theta}^*M\mathbf{e}-m||\mathbf{x}-\mathbf{x}^*||^2\\
&+\max_{i\in\mathcal{V}}\{l_i\}||\mathbf{x}-\mathbf{x}^*||||\mathbf{e}||\\
&+\bar{l}_1||\mathbf{e}||^2+\bar{l}_2||\mathbf{e}||||\mathbf{x}-\mathbf{x}^*||,
\end{aligned}
\end{equation}
where $\bar{l}_1=2||M||\sqrt{N}\max_{i\in\mathcal{V}}\{l_i\}$ and $\bar{l}_2=2||M||N\max_{i\in\mathcal{V}}\{l_i\}$. Note that the last inequality is obtained by noticing that
\begin{equation}
\begin{aligned}
&||2\mathbf{e}^TM \mathbf{1}_N\otimes \dot{\mathbf{x}}||\\
=&
||2\mathbf{e}^TM \mathbf{1}_N\otimes [\nabla_i f_i(\mathbf{y}_i)]_{vec}||\\
=&||2\mathbf{e}^TM \mathbf{1}_N\otimes [\nabla_i f_i(\mathbf{y}_i)-\nabla_i f_i(\mathbf{x})]_{vec}||\\
&+||2\mathbf{e}^TM \mathbf{1}_N\otimes[\nabla_i f_i(\mathbf{x})-\nabla_i f_i(\mathbf{x}^*)]_{vec}||\\
\leq &\bar{l}_1||\mathbf{e}||^2+\bar{l}_2||\mathbf{e}||||\mathbf{x}-\mathbf{x}^*||.
\end{aligned}
\end{equation}
Define
\begin{equation}
A=\left[
    \begin{array}{cc}
      m & -\frac{\bar{l}_2+\max_{i\in\mathcal{V}}\{l_i\}}{2} \\
      -\frac{\bar{l}_2+\max_{i\in\mathcal{V}}\{l_i\}}{2} & 2\lambda_{min}(\mathbf{\theta}^*)\lambda_{min}(MM)-\bar{l}_1 \\
    \end{array}
  \right],
\end{equation}
and let
\begin{equation}
\lambda_{min}(\mathbf{\theta}^*)>\frac{(\bar{l}_2+\max_{i\in\mathcal{V}}\{l_i\})^2+4m\bar{l}_1}{8m\lambda_{min}(MM)}.
\end{equation}
Then,
\begin{equation}
\dot{V}\leq -\lambda_{min}(A)||\mathbf{E}||^2,
\end{equation}
where $\lambda_{min}(A)>0$ and $\mathbf{E}=[(\mathbf{x}-\mathbf{x}^*)^T,\mathbf{e}^T]^T.$

Noticing that $\dot{V}\leq 0$ for all $t\geq 0,$ we get that $\theta_{ij}$ for all $i,j\in\mathcal{V}$ are bounded. In addition, by utilizing LaSalle's invariance principle \cite{Khailil02}, $||\mathbf{E}||\rightarrow 0,$ as $t\rightarrow \infty$. From \eqref{met}, we see that $\theta_{ij}$ is non-decreasing as it's time derivative is nonnegative for all $t\geq 0.$ Moreover, as $t\rightarrow \infty,$ $||\mathbf{e}||\rightarrow 0$ and $||\mathbf{x}-\mathbf{x}^*||\rightarrow 0,$ we see that $\dot{\theta}_{ij}\rightarrow 0$ as $t\rightarrow \infty.$   Hence, $\theta_{ij}$ for $i,j\in\mathcal{V}$ will converge to some finite values as $t\rightarrow \infty.$
\end{proof}

\subsection{Fully distributed strategy  through edge-based design}
In this section, an edge-based adaptive control strategy will be proposed by dynamically adjusting the weights on the edges of the communication graphs.

Motivated by \cite{YETAC2017} and \cite{LiTAC13}, design the Nash equilibrium seeking strategy as
\begin{equation}\label{meth2}
\begin{aligned}
\dot{x}_i=&-\nabla_if_i(\mathbf{y}_i),\\
\dot{y}_{ij}=&-\left(\sum_{k=1}^Na_{ik}c_{ik}(y_{ij}-y_{kj})+a_{ij}\bar{c}_{ij}(y_{ij}-x_j)\right),\\
\dot{c}_{ij}=&a_{ij}(\mathbf{y}_i-\mathbf{y}_j)^T(\mathbf{y}_i-\mathbf{y}_j),\\
\dot{\bar{c}}_{ij}=&c_{ij}(y_{ij}-x_j)^T(y_{ij}-x_j),
\end{aligned}
\end{equation}
where $i,j\in\mathcal{V}$ and $c_{ij}(0)=c_{ji}(0)$ for all $i,j\in\mathcal{V}.$

\begin{Remark}
Different from \eqref{met}, where the control gain on each player's accessible consensus error is designed to be adaptive, the weights on the edges of the communication graph are dynamically adjusted in \eqref{meth2}. From \eqref{meth2}, it can be seen that both $\dot{c}_{ij}$ and $\dot{\bar{c}}_{ij}$ are non-negative, indicating that $c_{ij}$ and $\bar{c}_{ij}$ are both non-decreasing.
\end{Remark}

Define $e_{ij}=y_{ij}-x_j$. Then, the method in \eqref{meth2} can be rewritten as
\begin{equation}
\begin{aligned}
\dot{x}_i=&-\nabla_if_i(\mathbf{e}_i+\mathbf{x}),\\
\dot{e}_{ij}=&-\left(\sum_{k=1}^Na_{ik}c_{ik}(e_{ij}-e_{kj})+a_{ij}\bar{c}_{ij}e_{ij}\right)-\dot{x}_j,\\
\dot{c}_{ij}=& a_{ij} (\mathbf{e}_i-\mathbf{e}_j)^T(\mathbf{e}_i-\mathbf{e}_j),\\
\dot{\bar{c}}_{ij}=& a_{ij}e_{ij}^Te_{ij}.
\end{aligned}
\end{equation}

\begin{Theorem}\label{the2}
Suppose that Assumptions \ref{ass1}-\ref{ass4} are satisfied. Then,
\begin{enumerate}
  \item The Nash equilibrium $\mathbf{x}^*$ is globally asymptotically stable;
  \item The estimation error $||\mathbf{y}-\mathbf{1}_N\otimes \mathbf{x}||\rightarrow 0$ as $t\rightarrow \infty$;
  \item The adaptive control gains $c_{ij}$ and $\bar{c}_{ij}$ for all $i,j\in\mathcal{V}$ converge to some finite values.
\end{enumerate}
\end{Theorem}
\begin{proof}
Define the Lyapunov candidate function as
\begin{equation}
V=V_1+V_2+V_3+V_4,
\end{equation}
where
\begin{equation}
\begin{aligned}
V_1&=\frac{1}{2}(\mathbf{x}-\mathbf{x}^*)^T(\mathbf{x}-\mathbf{x}^*),\\
V_2&=\frac{1}{2}\sum_{i=1}^N \mathbf{e}_i^T\mathbf{e}_i,\\
V_3&=\sum_{i=1}^N \sum_{j=1,j\neq i}^N \frac{(c_{ij}-c^*)^2}{4},
\end{aligned}
\end{equation}
and
\begin{equation}
V_4=\sum_{i=1}^N\sum_{j=1,j\neq i}^N \frac{(\bar{c}_{ij}-c^*)^2}{2},
\end{equation}
and $c^*$ is a positive constant to be determined later.
 Then,
 \begin{equation}
 \begin{aligned}
 \dot{V}_1=&-(\mathbf{x}-\mathbf{x}^*)^T [\nabla_if_i(\mathbf{e}_i+\mathbf{x})]_{vec}\\
 \leq & -m||\mathbf{x}-\mathbf{x}^*||^2+\max_{i\in\mathcal{V}}\{l_i\}||\mathbf{x}-\mathbf{x}^*||||\mathbf{e}||,
 \end{aligned}
 \end{equation}
 and
 \begin{equation}
 \dot{V}_2=-\sum_{i=1}^N \mathbf{e}_i^T \left(\sum_{k=1}^N a_{ik}c_{ik}(\mathbf{e}_i-\mathbf{e}_k)+ \bar{d}_i\mathbf{e}_i+\dot{\mathbf{x}}\right),
 \end{equation}
 where $\bar{d}_i$ is a diagonal matrix whose diagonal elements are $a_{i1}\bar{c}_{i1}$, $a_{i2}\bar{c}_{i2},\cdots$ $a_{iN}\bar{c}_{iN},$ successively.

Moreover,
\begin{equation}
\begin{aligned}
\dot{V}_3=&\frac{1}{2}\sum_{i=1}^N \sum_{j=1,j\neq i}^N (c_{ij}-c^*)\dot{c}_{ij}\\
=& \frac{1}{2}\sum_{i=1}^N \sum_{j=1,j\neq i}^N (c_{ij}-c^*) a_{ij}(\mathbf{e}_i-\mathbf{e}_j)^T(\mathbf{e}_i-\mathbf{e}_j)\\
=&\sum_{i=1}^N\sum_{j=1}^N (c_{ij}-c^*)a_{ij}\mathbf{e}_i^T(\mathbf{e}_i-\mathbf{e}_j),
\end{aligned}
\end{equation}
where the last inequality is obtained by utilizing the fact that $c_{ij}(t)=c_{ji}(t)$ for all $t\geq 0.$
In addition,
\begin{equation}
\dot{V}_4=\sum_{i=1}^N\sum_{j=1,j\neq i}^N(\bar{c}_{ij}-c^*)a_{ij}e_{ij}^Te_{ij}.
\end{equation}
Hence,
\begin{equation}
\begin{aligned}
\dot{V}\leq &-m||\mathbf{x}-\mathbf{x}^*||^2+\max_{i\in\mathcal{V}}\{l_i\}||\mathbf{x}-\mathbf{x}^*||||\mathbf{e}||\\
&-\sum_{i=1}^N \mathbf{e}_i^T \left(\sum_{k=1}^N a_{ik}c_{ik} (\mathbf{e}_i-\mathbf{e}_k)+ \bar{d}_i\mathbf{e}_i+\dot{\mathbf{x}}\right)\\
&+\sum_{i=1}^N\sum_{j=1}^N (c_{ij}-c^*)a_{ij}\mathbf{e}_i^T(\mathbf{e}_i-\mathbf{e}_j)\\
&+\sum_{i=1}^N\sum_{j=1,j\neq i}^N(\bar{c}_{ij}-c^*)a_{ij}e_{ij}^Te_{ij}\\
\leq & -m||\mathbf{x}-\mathbf{x}^*||^2+\max_{i\in\mathcal{V}}\{l_i\}||\mathbf{x}-\mathbf{x}^*||||\mathbf{e}||\\
&-c^* \mathbf{e}^T M\mathbf{e}+\sum_{i=1}^N \mathbf{e}_i^T \dot{\mathbf{x}},
\end{aligned}
\end{equation}
where
\begin{equation}
\begin{aligned}
&\left|\left|\sum_{i=1}^N \mathbf{e}_i^T \dot{\mathbf{x}}\right|\right|\leq  \left|\left|\sum_{i=1}^N \mathbf{e}_i^T[\nabla_if_i(\mathbf{y}_i)-\nabla_if_i(\mathbf{x})]_{vec}\right|\right|\\
&+\left|\left|\sum_{i=1}^N \mathbf{e}_i^T\left[\nabla_if_i(\mathbf{x})-\nabla_if_i(\mathbf{x}^*)\right]_{vec}\right|\right|\\
\leq& \max_{i\in\mathcal{V}}\{l_i\}||\mathbf{e}||^2+\max_{i\in\mathcal{V}}\{l_i\}\sqrt{N}||\mathbf{e}||||\mathbf{x}-\mathbf{x}^*||.
\end{aligned}
\end{equation}
Define
\begin{equation}
A=\left[
   \begin{array}{cc}
     m & -\frac{\max_{i\in\mathcal{V}}\{l_i\}(1+\sqrt{N})}{2} \\
     -\frac{\max_{i\in\mathcal{V}}\{l_i\}(1+\sqrt{N})}{2} & c^*\lambda_{min}(M)-\max_{i\in\mathcal{V}}\{l_i\} \\
   \end{array}
 \right],
\end{equation}
and choose
\begin{equation}
c^*>\frac{(\max_{i\in\mathcal{V}}\{l_i\}(1+\sqrt{N}))^2}{4m \lambda_{min}(M)}+\frac{\max_{i\in\mathcal{V}}\{l_i\}}{\lambda_{min}(M)},
\end{equation}
then, $\lambda_{min}(A)>0$ and
\begin{equation}
\dot{V}\leq -\lambda_{min}(A)||\mathbf{E}||^2,
\end{equation}
$\mathbf{E}=[(\mathbf{x}-\mathbf{x}^*)^T,(\mathbf{y}-\mathbf{1}_N\otimes \mathbf{x})^T]^T.$
Hence, $\dot{V}\leq 0,$ indicating that $\mathbf{E}$, $c_{ij}$ and $\bar{c}_{ij}$ stay bounded for all $i,j\in\mathcal{V}$. Moreover, by LaSalle's invariance principle \cite{Khailil02},  $||\mathbf{x}-\mathbf{x}^*||\rightarrow 0$ and $||\mathbf{y}-\mathbf{1}_N\otimes \mathbf{x}||\rightarrow 0$
as $t\rightarrow \infty.$ In addition, as $||\mathbf{x}-\mathbf{x}^*||\rightarrow 0$ and $||\mathbf{y}-\mathbf{1}_N\otimes \mathbf{x}||\rightarrow 0$, $\dot{c}_{ij}\rightarrow 0$ and $\dot{\bar{c}}_{ij}\rightarrow 0$ as $t\rightarrow \infty.$ Therefore, $c_{ij}$ and $\bar{c}_{ij}$ for all $i,j\in\mathcal{V}$ would converge to some finite values.
 \end{proof}
\begin{Remark}
From Theorems \ref{the1}-\ref{the2}, we see that both the proposed node-based adaptive Nash equilibrium seeking algorithm and edge-based adaptive Nash equilibrium seeking algorithm can achieve fully distributed Nash equilibrium seeking. Moreover, compared with the method in \cite{YETAC2017}, it is clear that the requirement on the quantification of the singular perturbation parameter is removed in the proposed methods.
\end{Remark}

\section{Extensions to switching communication conditions}\label{time}
In this section, we suppose that the communication graph is switching among a set of undirected and connected communication graphs. Let $p$ be the number of possible undirected graphs in the communication condition and define $\sigma(t):[0,\infty)\rightarrow 1,2,\cdots,p$ as a switching signal that maps the time sequence to the set of possible graphs. Suppose that the communication graph is fixed for $t\in[t_k,t_{k+1})$ where $k=0,1,2,\cdots$ and  there exists a positive constant $\epsilon$ such that $t_{k+1}-t_k\geq \epsilon$ for all $k=0,1,2,\cdots.$ For notational convenience, denote the communication graph at time $t$ as $\mathcal{G}^{\sigma(t)}$. Correspondingly, the adjacency matrix, Laplacian matrix and the $(i,j)$th entry of the adjacency matrix are denoted as $\mathcal{A}^{\sigma(t)},$ $\mathcal{L}^{\sigma(t)}$ and $a_{ij}^{\sigma(t)},$ respectively.
\begin{Assumption}\label{ass5}
For each $i\in\{1,2,\cdots,p\},$ the communication graph $\mathcal{G}^i$ is undirected and connected.
\end{Assumption}

Accordingly, the edge-based seeking strategy is adapted as
\begin{equation}\label{met3}
\begin{aligned}
\dot{x}_i=&-\nabla_if_i(\mathbf{y}_i),\\
\dot{y}_{ij}=&-\left(\sum_{k=1}^Na_{ik}^{\sigma(t)}c_{ik}(y_{ij}-y_{kj})+a_{ij}^{\sigma(t)}\bar{c}_{ij}(y_{ij}-x_j)\right),\\
\dot{c}_{ij}=&a_{ij}^{\sigma(t)}(\mathbf{y}_i-\mathbf{y}_j)^T(\mathbf{y}_i-\mathbf{y}_j),\\
\dot{\bar{c}}_{ij}=&a_{ij}^{\sigma(t)}(y_{ij}-x_j)^T(y_{ij}-x_j),
\end{aligned}
\end{equation}
for $i,j\in\mathcal{V}.$

Then, the following result can be obtained.
\begin{Theorem}\label{the3}
Suppose that Assumptions \ref{ass1}-\ref{ass3} and \ref{ass5} are satisfied. Then,
\begin{enumerate}
  \item The Nash equilibrium $\mathbf{x}^*$ is globally asymptotically stable.
  \item The players' estimation error $||\mathbf{y}-\mathbf{1}_N\otimes\mathbf{x}||\rightarrow 0$ as $t\rightarrow \infty.$
  \item The adaptive control gains $c_{ij}$ and $\bar{c}_{ij}$ for all $i,j\in\mathcal{V}$ would converge to some finite values.
\end{enumerate}
\end{Theorem}
\begin{Proof}
Define the Lyapunov candidate function as
\begin{equation}
\begin{aligned}
V=&\frac{1}{2}(\mathbf{x}-\mathbf{x}^*)^T(\mathbf{x}-\mathbf{x}^*)+\frac{1}{2}\sum_{i=1}^N \mathbf{e}_i^T\mathbf{e}_i\\
&+\sum_{i=1}^N \sum_{j=1,j\neq i}^N \frac{(c_{ij}-c^*)^2}{4}+\sum_{i=1}^N\sum_{j=1,j\neq i}^N \frac{(\bar{c}_{ij}-c^*)^2}{2},
\end{aligned}
\end{equation}
where $c^*$ is a positive constant to be determined later.
Then, following the analysis in the proof of Theorem \ref{the2},
it can be obtained that
\begin{equation}
\begin{aligned}
\dot{V}\leq &-m||\mathbf{x}-\mathbf{x}^*||^2+\max_{i\in\mathcal{V}}\{l_i\}||\mathbf{x}-\mathbf{x}^*||||\mathbf{e}||\\
&-c^* \mathbf{e}^T (\mathcal{L}^{\sigma(t)}\otimes I_{N\times N}+\mathcal{A}_0^{\sigma(t)})\mathbf{e}\\
&+\max_{i\in\mathcal{V}}\{l_i\}||\mathbf{e}||^2+\sqrt{N}\max_{i\in\mathcal{V}}\{l_i\}||\mathbf{e}||||\mathbf{x}-\mathbf{x}^*||,
\end{aligned}
\end{equation}
where $\mathcal{A}_0^{\sigma(t)}=\text{diag}\{a_{ij}^{\sigma(t)}\}.$
Define $\underline{\lambda}$ as the minimum eigenvalue of $\mathcal{L}^{\sigma(t)}\otimes I_{N\times N}+\mathcal{A}_0^{\sigma(t)}$ for all possible combinations of
$\mathcal{L}^{\sigma(t)}$ and $\mathcal{A}_0^{\sigma(t)}$. Then,

\begin{equation}
\begin{aligned}
\dot{V}\leq &-m||\mathbf{x}-\mathbf{x}^*||^2+\max_{i\in\mathcal{V}}\{l_i\}||\mathbf{x}-\mathbf{x}^*||||\mathbf{e}||\\
&-c^*\underline{\lambda}||\mathbf{e}||^2+\max_{i\in\mathcal{V}}\{l_i\}||\mathbf{e}||^2\\
&+\sqrt{N}\max_{i\in\mathcal{V}}\{l_i\}||\mathbf{e}||||\mathbf{x}-\mathbf{x}^*||.
\end{aligned}
\end{equation}
Then, by choosing $c^*>\frac{(\max\{l_i\}(1+ \sqrt{N}))^2}{4m\underline{\lambda}}+\frac{\max\{l_i\}}{\underline{\lambda}},$ it can be obtained that
\begin{equation}
\dot{V}\leq -\lambda_{min}(A)||\mathbf{E}||^2,
\end{equation}
where $\lambda_{min}(A)>0$ and $A=\left[
           \begin{array}{cc}
             m & -\frac{\max\{l_i\}(1+\sqrt{N})}{2} \\
             -\frac{\max\{l_i\}(1+\sqrt{N})}{2} & c^*\underline{\lambda}-\max_{i\in\mathcal{V}}\{l_i\} \\
           \end{array}
         \right]$.
Hence, by further utilizing the LaSalle-Yoshizawa theorem \cite{Krstic}, the conclusions can be obtained.
\end{Proof}

\begin{Remark}
In \cite{YECYBER2018}, a Nash equilibrium seeking was established for networked games under time-varying communication networks. However, it was shown that the singular perturbation parameter should be neither to large nor too small to ensure the convergence of the presented method. However, it is challenging to explicitly calculate the lower-bound and the upper-bound of the singular perturbation parameter. Different from the method in \cite{YECYBER2018}, the adaptive method in this section is fully distributed. Hence, it is more suitable for practical implementation in distributed systems.
\end{Remark}

\section{A Numerical example}\label{num_ex}

In this section, we consider the connectivity control game for a network of $5$ mobile sensor networks considered in \cite{YECDC19}.  In the considered game, $x_i\in\mathbb{R}^2.$ For notational convenience, denote $x_i$ as $x_i=[x_{i1},x_{i2}]^T.$ Moreover, the players' objective functions are defined as
\begin{equation}
\begin{aligned}
f_1(\mathbf{x})=&x_1^Tx_1+x_{11}+x_{12}+1+||x_1-x_2||^2,\\
f_2(\mathbf{x})=&2x_2^Tx_2+2x_{21}+2x_{22}+2+||x_2-x_3||^2,\\
f_3(\mathbf{x})=&3x_3^Tx_3+3x_{31}+3x_{32}+3+||x_3-x_2||^2,\\
f_4(\mathbf{x})=&4x_4^Tx_4+4x_{41}+4x_{42}+4+||x_4-x_2||^2\\
&+||x_4-x_5||^2,\\
f_5(\mathbf{x})=&5x_4^Tx_4+5x_{51}+5x_{52}+5+||x_5-x_1||^2,
\end{aligned}
\end{equation}
by which the Nash equilibrium is unique and is $x_{ij}^*=0.5,$ for all $i\in \{1,2,\cdots,5\},j\in\{1,2\}$ \cite{YECDC19}.
\begin{figure}[!htp]
\centering
\scalebox{0.4}{\includegraphics{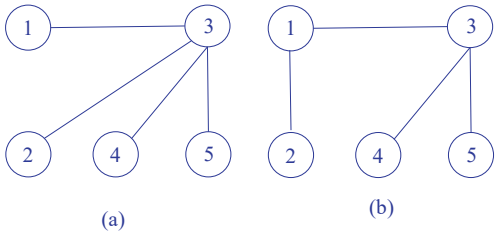}}
\vspace{-45mm}
\caption{The communication graphs.}\label{graph}
\end{figure}

In the rest, we provide simulation results for the methods in
\eqref{met}, \eqref{meth2} and \eqref{met3}, successively. Note that as $x_i\in\mathbb{R}^2,$ $y_{ij}\in\mathbb{R}^2$ as well. For notational clarity, we write it as $y_{ij}=[y_{ij1},y_{ij2}]^T$ in the remainder of the paper.
\subsection{The node-based adaptive Nash equilibrium seeking strategy under a fixed communication graph}
In this section, we consider that the players' actions are updated according to \eqref{met} under the communication graph in Fig. \ref{graph} (a).

In the simulation, $\mathbf{x},\mathbf{y},$ and $\theta_{ij},i,j\in\mathcal{V}$ are initialized randomly between $-20$ to $20$. In addition, $\gamma_{ij}=1$ for all $i,j\in\mathcal{V}.$ The simulation results generated by the proposed method in \eqref{met} are given in Figs. \ref{adaptive1_state}-\ref{adaptive1_esti}. The horizontal and vertical coordinates of Fig. \ref{adaptive1_state} are $x_{i1}$ and $x_{i2},$ respectively. From Fig. \ref{adaptive1_state}, we see that though the variables are randomly initialized, the proposed method enables the sensors' position trajectories to converge to the Nash equilibrium. In addition, Fig. \ref{adaptive1_para} illustrates $\theta_{ij}$ for $i,j\in\{1,2,\cdots,5\},$ from which it is clear that $\theta_{ij}$ for all $i,j\in\{1,2,\cdots,5\}$ converge to some finite values. The  plots of $y_{ij1}$ versus $y_{ij2}$ for $j\in\{1,2,\cdots,5\}$ are given in the $5$ sub-figures of Fig. \ref{adaptive1_esti} for players $1$-$5$, respectively. From Fig. \ref{adaptive1_esti}, we see that $\mathbf{y}_i$ for all $i\in\{1,2,\cdots,5\}$ also converge to the Nash equilibrium, thus verifying Theorem \ref{the1}.

\begin{figure}[!htp]
\centering
\scalebox{0.55}{\includegraphics{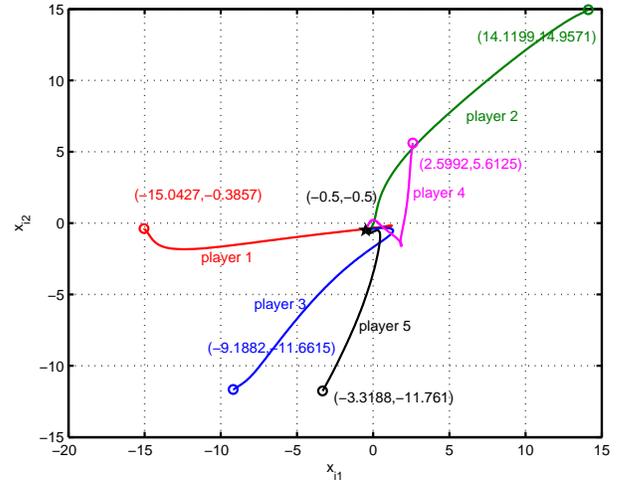}}
\caption{The plots of the players' state trajectories, i.e., the plots of $x_{i1}$ versus $x_{i2}$ for $i=\{1,2,\cdots,5\}$ generated by \eqref{met}.}\label{adaptive1_state}
\end{figure}

\begin{figure}[!htp]
\centering
\scalebox{0.55}{\includegraphics{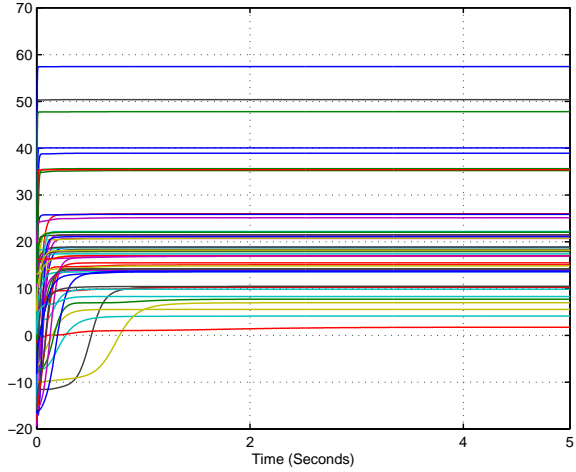}}
\caption{The trajectories of $\theta_{ij}$ for  $i,j\in \{1,2,\cdots,5\}$ generated by \eqref{met}.}\label{adaptive1_para}
\end{figure}

\begin{figure}[!htp]
\centering
\scalebox{0.55}{\includegraphics{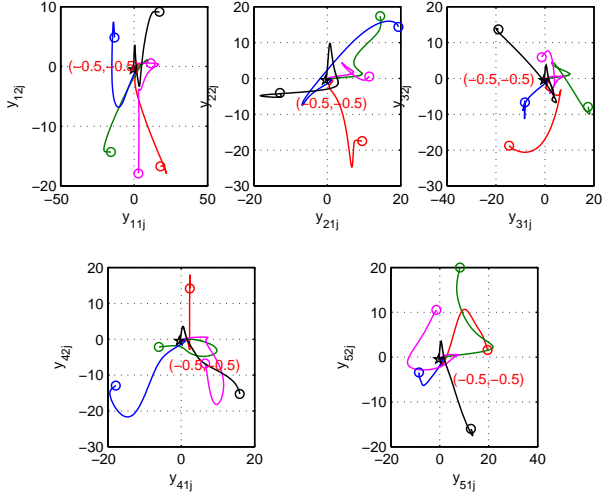}}
\caption{The plots of $y_{ij1}$ versus $y_{ij2}$ for $i,j\in\{1,2,\cdots,5\}$ generated by \eqref{met}, respectively.}\label{adaptive1_esti}
\end{figure}

\subsection{The edge-based adaptive Nash equilibrium seeking strategy under a fixed communication graph}
In this section, we simulate the edge-based adaptive Nash equilibrium seeking algorithm given in \eqref{meth2}. Correspondingly, the communication graph is given in Fig. \ref{graph} (a).

In the simulation, $\mathbf{x},$ $\mathbf{y}$ are initialized randomly between $-20$ and $20.$ Moreover, $c_{ij},\bar{c}_{ij}$ for $i,j\in\{1,2,\cdots,5\}$ are initialized at $1$. Fig. \ref{adaptive2_state} plots the sensors' position  trajectories. The horizontal and vertical coordinates of  Fig. \ref{adaptive2_state} represents the values of $x_{i1}$ and $x_{i2}$ for $i\in\{1,2,\cdots,5\},$ respectively. From Fig.  \ref{adaptive2_state}, it can be seen that the sensors' position trajectories converge to the actual Nash equilibrium of the game.

Moreover, Fig. \ref{adaptive2_para} plots $c_{ij}$ and $\bar{c}_{ij}$ from which it can be seen that they converge to some finite values.
The $5$ sub-figures in Fig. \ref{adaptive2_esti} plot $y_{ij1}$ versus $y_{ij2}$, $j\in\{1,2,\cdots,5\}$ for players $1$-$5,$ respectively. From Fig.  \ref{adaptive2_esti}, we see that $\mathbf{y}_i$ for all $i\in\{1,2,\cdots,5\}$ converge to the Nash equilibrium $\mathbf{x}^*$ as well. Hence, Theorem \ref{the2} is numerically verified.
\begin{figure}[!htp]
\centering
\scalebox{0.55}{\includegraphics{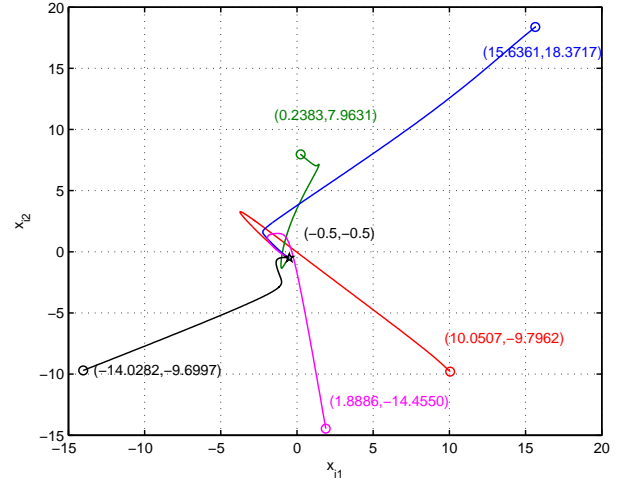}}
\caption{The plots of the players' state trajectories, i.e., the plots of $x_{i1}$ versus $x_{i2}$ for $i=\{1,2,\cdots,5\}$, generated by \eqref{meth2}.}\label{adaptive2_state}
\end{figure}

\begin{figure}[!htp]
\centering
\scalebox{0.55}{\includegraphics{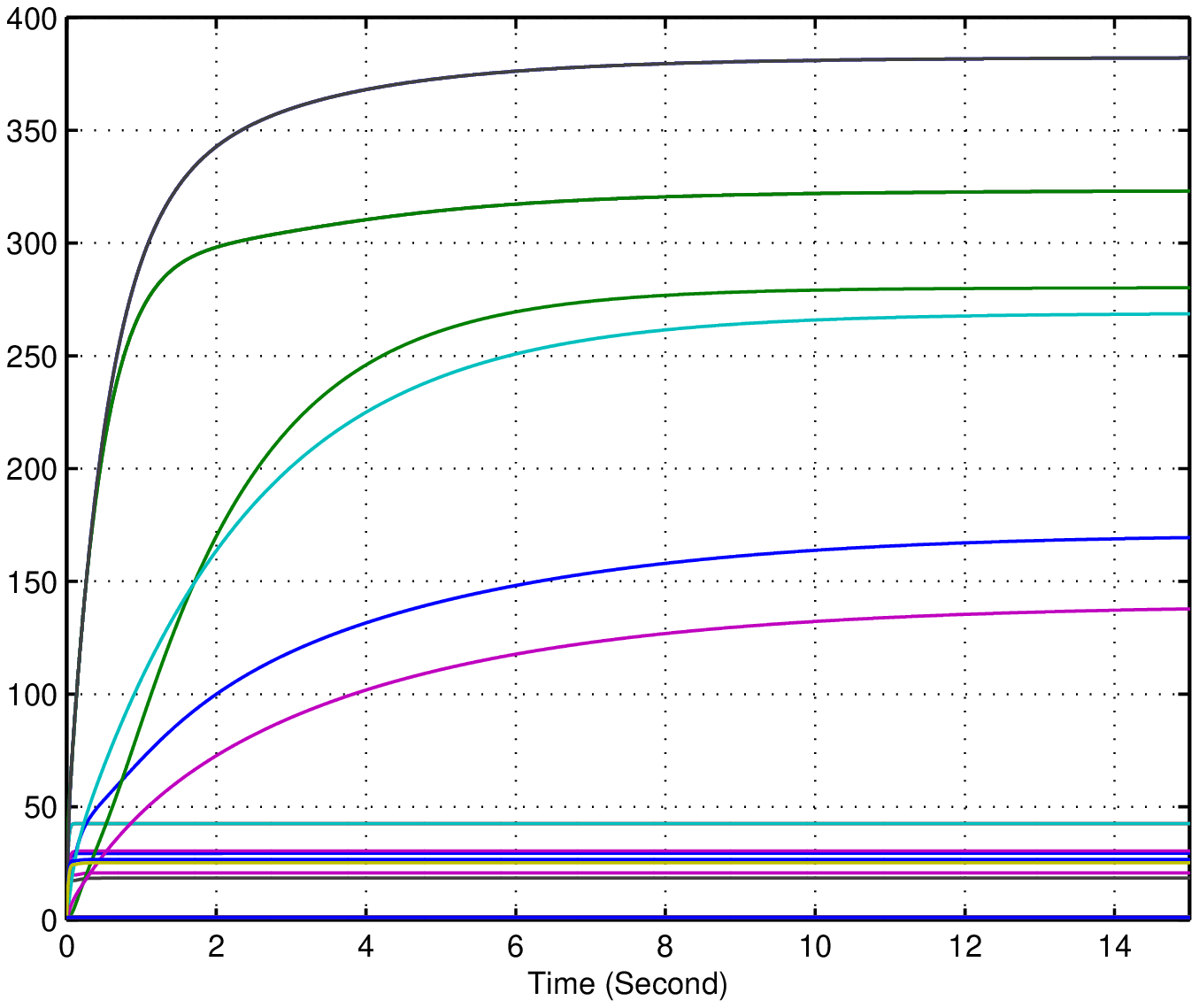}}
\caption{The trajectories of $c_{ij}$ and $\bar{c}_{ij}$ for  $i,j\in \{1,2,\cdots,5\}$ generated by \eqref{meth2}.}\label{adaptive2_para}
\end{figure}

\begin{figure}[!htp]
\centering
\scalebox{0.55}{\includegraphics{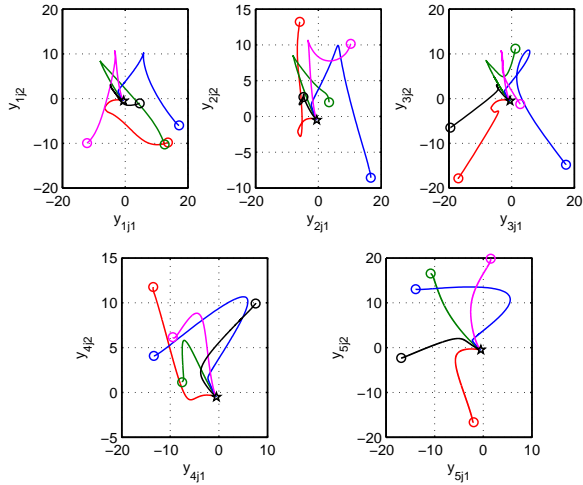}}
\caption{The plots of $y_{ij1}$ versus $y_{ij2}$ for $i,j\in\{1,2,\cdots,5\}$, generated by \eqref{meth2}, respectively.}\label{adaptive2_esti}
\end{figure}

\subsection{Fully distributed Nash equilibrium seeking under switching communication topologies}
In this section, we simulate the adaptive Nash equilibrium seeking algorithm given in \eqref{met3} under switching communication conditions. Similar to the previous subsections, $\mathbf{x},$ $\mathbf{y}$ are initialized randomly between $-20$ and $20.$ Moreover, $c_{ij},\bar{c}_{ij}$ for $i,j\in\{1,2,\cdots,5\}$ are initialized at $1$. In addition, we suppose that for $t<0.5$ and $5<t<8,$ the players communicate via the graph in Fig. \ref{graph} (a). For other time sequences, the players communicate via the graph in Fig. \ref{graph} (b).
The simulation results are given in Figs. \ref{adaptive3_state}-\ref{adaptive3_esti}.  Fig. \ref{adaptive3_state} plots the sensors' trajectories. The horizontal and vertical coordinates of  Fig. \ref{adaptive3_state} are the values of $x_{i1}$ and $x_{i2}$ for $i\in\{1,2,\cdots,5\},$ respectively. From Fig.  \ref{adaptive3_state}, it can be seen that the sensors' position trajectories converge to the actual Nash equilibrium of the game. Fig. \ref{adaptive3_para} shows $c_{ij}$ and $\bar{c}_{ij}$ from which it can be seen that they converge to some finite values. Moreover, the $5$ sub-figures in Fig. \ref{adaptive3_esti} plot $y_{ij1}$ versus $y_{ij2}$, $j\in\{1,2,\cdots,5\}$ for players $1$-$5,$ respectively. From Fig.  \ref{adaptive3_esti}, we see that $\mathbf{y}_i$ for all $i\in\{1,2,\cdots,5\}$ converge to the Nash equilibrium $\mathbf{x}^*$. Hence, Theorem \ref{the3} is numerically verified.

\begin{figure}[!htp]
\centering
\scalebox{0.55}{\includegraphics{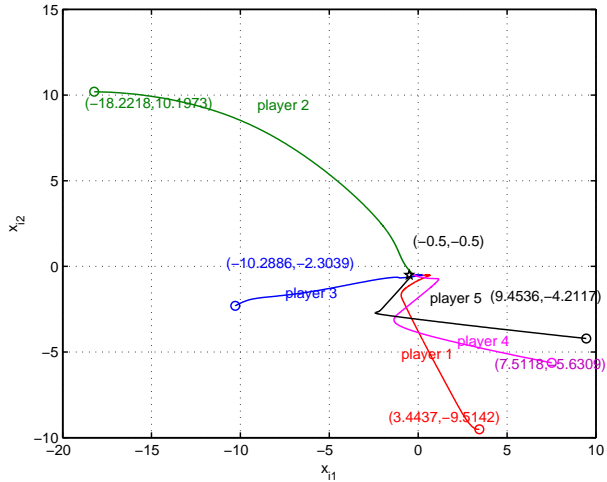}}
\caption{The plots of the players' state trajectories, i.e., the plots of $x_{i1}$ versus $x_{i2}$ for $i=\{1,2,\cdots,5\}$, generated by \eqref{met3}.}\label{adaptive3_state}
\end{figure}

\begin{figure}[!htp]
\centering
\scalebox{0.55}{\includegraphics{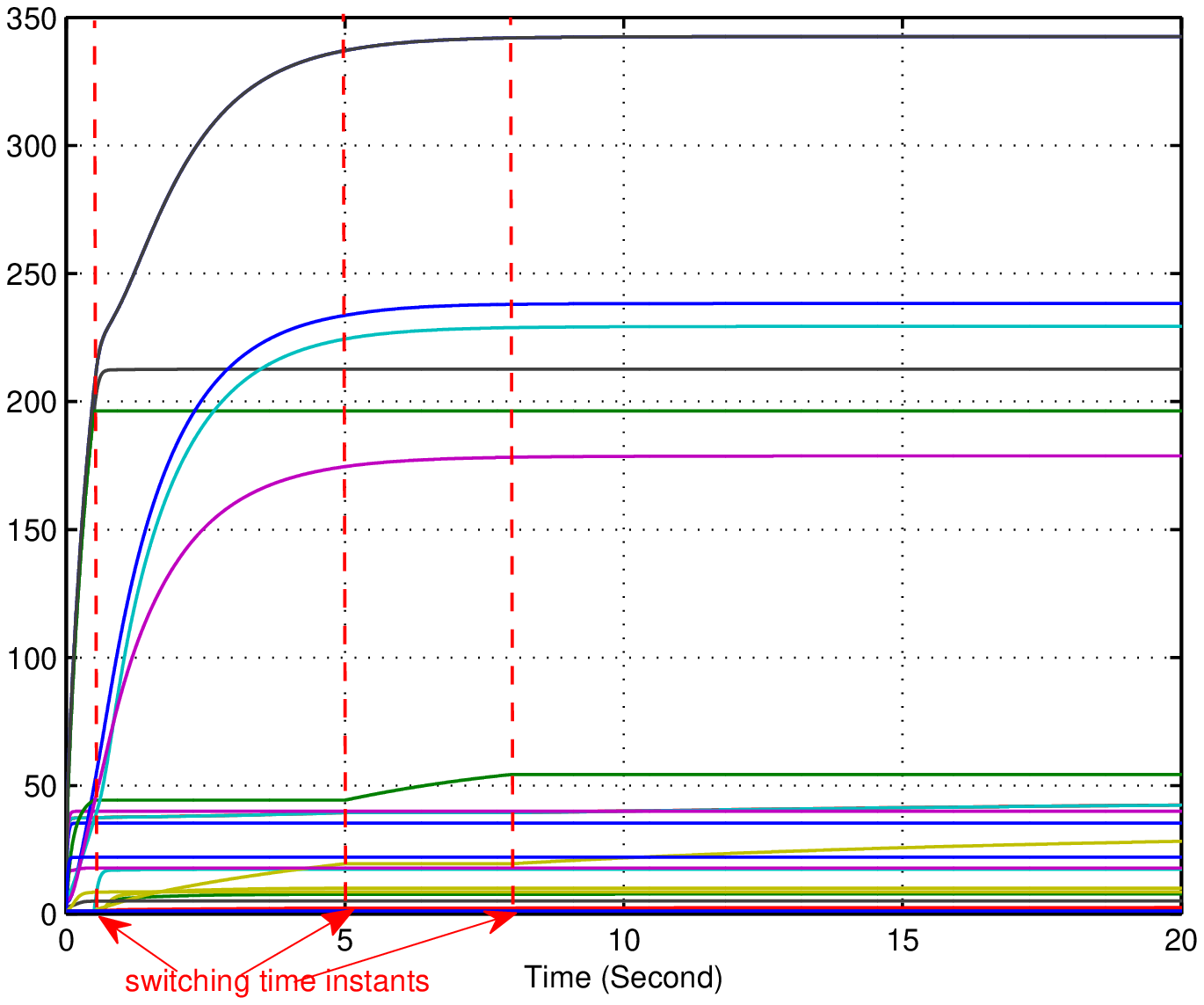}}
\caption{The trajectories of $c_{ij}$ and $\bar{c}_{ij}$ for  $i,j\in \{1,2,\cdots,5\}$ generated by \eqref{met3}.}\label{adaptive3_para}
\end{figure}

\begin{figure}[!htp]
\centering
\scalebox{0.55}{\includegraphics{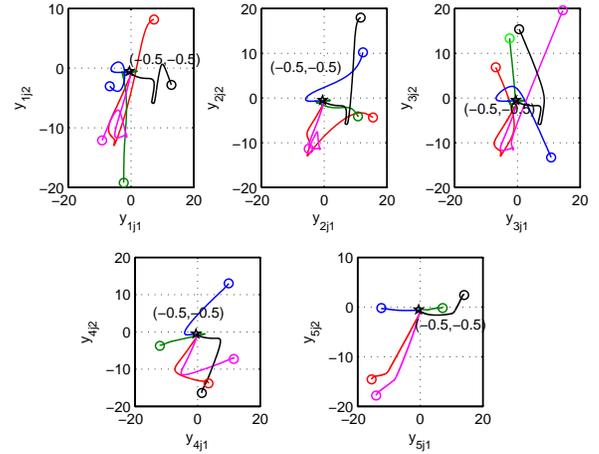}}
\caption{The plots of $y_{ij1}$ versus $y_{ij2}$ for $i,j\in\{1,2,\cdots,5\}$, generated by \eqref{met3},  respectively.}\label{adaptive3_esti}
\end{figure}
\section{Conclusions}\label{conc}
This paper develops fully distributed Nash equilibrium seeking algorithms for networked game. A node-based adaptive algorithm and an edge-based adaptive algorithm are proposed. The node-based adaptive algorithm achieves the distributed strategy design via adjusting the weight on each player's accessible consensus error adaptively. The edge-based adaptive algorithm dynamically adjusts the weights on the edges of the communication graph to achieve fully distributed Nash equilibrium seeking. Both algorithms are theoretically validated via Lyapunov stability analysis. In addition, to highlight the extensibility of the proposed method under time-varying communication networks, the edge-based algorithm is adapted to accommodate switching communication conditions in which the communication network switches among a set of undirected and connected communication graphs.

\end{document}